\documentclass[12pt]{amsart}
\usepackage{amsthm}
\usepackage{enumerate}
\usepackage{float}
\usepackage{dcolumn}
\usepackage{rotating}
\usepackage{geometry}
\usepackage[british]{babel}
\theoremstyle{plain}

\newtheorem{Pro}{\bf Proposition}
\newtheorem{Cor}[Pro]{\bf Corollary}
\newtheorem{Conj}[Pro]{\bf Conjecture}
\theoremstyle{definition}
\newtheorem{Def}{\bf Definition}
\newcommand{\cU}{\mathcal U}
\newcommand{\cC}{\mathcal C}
\newcommand{\cD}{\mathcal D}
\newcommand{\ccC}{\widetilde{\cC}}
\newcommand{\cM}{\mathcal M}
\newcommand{\cO}{\mathcal O}
\newcommand{\PG}{\operatorname{PG}}

\ifx\pdfoutput\undefined \newcount\pdfoutput \fi
\ifcase\pdfoutput
\usepackage[active]{srcltx}
\else
\usepackage{times}
\usepackage{fontenc}[T1]
\fi
\newenvironment{framed}{}{}
\floatstyle{ruled}
\floatname{alg}{Algorithm}
\newfloat{alg}{hbp}{lop}[section]
\newcommand{\begalg}{\begin{alg}[ht]\tt\begin{center}\begin{framed}%
        \begin{minipage}{10cm}}
        \newcommand{\enalg}{\end{minipage}\end{framed}\end{center}\end{alg}}
\title[Ovoids: a computational approach]{Looking for ovoids of the %
  Hermitian surface: a computational approach}
\author{Luca Giuzzi}
\address{Luca Giuzzi \\
  Dipartimento di Matematica \\
  Fa\-col\-t\`a di Ingegneria \\
  U\-ni\-ver\-si\-t\`a degli studi di Brescia \\
  Via Valotti 9 \\
  25133 Brescia (Italy)}
\email{giuzzi@dmf.unicatt.it}
\urladdr{http://www.dmf.unicatt.it/\textasciitilde giuzzi}
\usepackage{layout}
\begin{document}
\begin{abstract}
  In this note we introduce a computational approach to the
  construction of ovoids of the Hermitian surface and
  present some related experimental results.  
\end{abstract}
\maketitle
\tableofcontents
\section*{Introduction}
Let $q$ be a prime power and denote by $\cU$ the non--degenerate
Hermitian surface of $\PG(3,q^2)$.
A {\em generator} of $\cU$ is a maximal linear subspace contained
in $\cU$ --- in the case of Hermitian surfaces, a generator is
 a line.
A {\em Hermitian cap} $\cC$  is a subset of $\cU$
which is met by any generator  in at most one point.
A Hermitian cap is a {\em Hermitian ovoid} if and only if it is
met by any generator of in exactly one point.
\par
A Hermitian cap $\cC$  is usually not a cap  of the space $\PG(3,q^2)$.
In this paper, `caps' and `ovoids'
will always be assumed to mean Hermitian caps and Hermitian ovoids.
\par
An example of ovoid is provided by the intersection of the
Hermitian surface $\cU$ with any non--tangent plane;
however, several different constructions are 
known which lead to projectively inequivalent Hermitian ovoids,
see for instance \cite{BW}, \cite{Th2}, \cite{PT}.
\par
A  Hermitian cap which is maximal with 
respect to inclusion is said to be {\em complete}.
Hermitian ovoids are clearly complete; yet, there exist also complete
caps which are not ovoids.
\par
As a matter of fact, see \cite{ob}, it is known that if $\ccC$ is
 a complete cap,
then
\[ q^2+1\leq|\ccC|\leq q^3+1, \]
and both bounds are sharp: in particular,
$\ccC$ is an ovoid if and only if $|\ccC|=q^3+1$.
\par
Any maximal curve $\cD$ embedded in $\cU$ and different from the
 Hermitian curve, see \cite{KT},
provides an example of a cap of size approximately
$\frac{1}{4}(q^3-q^2)$.
For $q$ even, these caps are always complete;
when $q$ is odd, this is not the 
case: in fact, the set of the points of $\cD$ is usually contained in
an ovoid.
\par
The original motivation for this work has been to construct
some tools in order to help with the investigation of the relationship between
these partial caps and their completions.
However,
it is currently an open problem to determine
the spectrum of cardinalities of complete caps of
the Hermitian surface.
Numerical evidence suggests that, at least for $q$ prime, there should exist
complete caps of cardinality $t$ for almost all
values $q^2+1\leq t\leq q^3-q+1$.
It appears also that  complete caps are not
evenly distributed within this range.
\par
In Section 1, we introduce a strategy to look for complete caps
of the Hermitian surface; in Section 2, some improvements on the
basic algorithm are suggested;
in Section 3, we provide
the results of our computations for the cases $q=5$ and $q=7$;
these result lead us to 
conjecture that the size of the second largest complete cap 
is $q^3-q+1$.

\section{Basic completion strategy}
A {\em generator} of the surface $\cU$
 is a line of $\PG(3,q^2)$ completely included in $\cU$.
For any $x\in\cU$, denote by $Gx$ the set
of all generators of $\cU$ passing through $x$.
If we write by $T_x\cU$ the tangent plane at $x$ to $\cU$,
then the set $Gx$ may be determined as
\[ Gx=T_x\cU\cap\cU. \]
A point $p\in\cU$ is 
{\em covered} by
a set $\cM\subseteq\cU$ whenever 
\[ PP\cap\cM\neq\emptyset. \]
The set of points being covered by $\cM$ is written
as $G\cM$. It is straightforward to show that
\[ G\cM=\bigcup_{x\in\cM}Gx. \]
\par
\begin{Pro}
  \label{search}
  Let  $\cC$ be a cap of $\cU$; take $x\in\cU\setminus\cC$.
  Then, the set
  $\widetilde{\cC}=\cC\cup\{x\}$ is a cap of $\cU$ if
  and only if $x\not\in G\cC$.
\end{Pro} 
\begin{proof}
  If $x\in G\cC$, then there exists a generator $L$ of
  $\cU$ such that $x\in L$ and $L\cap\cC\neq\emptyset$.
  Since $x\not\in\cC$, it follows that 
  \[ |L\cap\widetilde{\cC}|=2; \]
  hence, in this case, $\widetilde{\cC}$ is not a cap.
  \par
  Assume now $x$ not to be covered by $\cC$
  and  let $L$ be any generator of $\cU$.
  If $x\in L$, then $L\cap\cC=\emptyset$; hence,
  $|L\cap\widetilde{\cC}|=1$.
  On the other hand, if $x\not\in L$, then
  \[ L\cap\widetilde{\cC}=L\cap\cC, \]
  which yields $|L\cap\tilde{\cC}|\leq 1$.
  It follows that any generator $L$ of $\cU$
  meets $\widetilde{\cC}$ in at most one point ---
  that is, $\widetilde{\cC}$ is a cap.
\end{proof}

For any given cap $\cC$,
Algorithm \ref{alg1} provides
a complete cap $\widetilde{\cC}$ with
$\cC\subseteq\widetilde{\cC}$.

\begalg
\caption{Basic completion algorithm}
\label{alg1}
\begin{tabular}{ll}
  {\bf Input:} & a cap $\cC$; \\
  {\bf Output:}& a complete cap $\widetilde{\cC}$.
\end{tabular}
\vskip.2cm
Complete($\cC$):=
\begin{enumerate}
\item Compute the set $\cM$ of points of $\cU$ not covered by $\cC$;
\item If $\cM=\emptyset$, return $\cC$ and exit;
\item Pick a random element $x\in\cM$;
\item $\cC\leftarrow(\cC\cup\{x\})$;
\item If $|\cC|=q^3+1$, return $\cC$ and exit;
\item Compute the set $\cM'=(\cM\setminus Gx)$;
\item $\cM\leftarrow\cM'$;
\item Go back to step (2).
\end{enumerate}
\enalg

This algorithm is guaranteed to complete in {\em at most}
$q^3+1-|\cC|$ iterations.

An efficient way to implement step {\tt (6)} is to compute
$\cM'$ as the set of points of $\cM$ which are not conjugate
to $x$ according to the unitary polarity induced by $\cU$.

\section{Large and small completions}
For any partial cap $\cC$,
Algorithm \ref{alg1} determines a complete cap $\widetilde{\cC}$
with $\cC\subseteq\widetilde{\cC}$.
However, a small cap $\cC$ usually admits several different completions,
as it can be seen from the tables of Section \ref{sec:31}. In fact,
even completions with the same cardinality needs not be
projectively equivalent, as it can be seen in the case of ovoids.
\begin{Def}
  A completion $\widetilde{\cC}$ of $\cC$ is {\em minimum} if, for any
  complete cap $\cD$ such that $\cC\subseteq\cD$,
  \[ |\widetilde{\cC}|\leq|\cD|; \]
  a completion $\widetilde{\cC}$ is {\em maximum} if
  \[ |\widetilde{\cC}|\geq|\cD|, \]
  for any complete cap $\cD$ with $\cC\subseteq\cD$.
  We call a completion $\ccC$ {\em optimal} if
  $\ccC$ it is either maximum or minimum.
\end{Def}
If there is a completion $\ccC$ of $\cC$ such
that 
\[ |\ccC|\leq|\cC|+1, \]
then, clearly, $\ccC$ is a {\em minimum} completion of $\cC$.
Likewise, if there is an ovoid $\cO$ containing $\cC$, then again $\cO$ is
a {\em maximum} completion of $\cC$.

To determine the size of the optimal completions of a given partial
cap is, in general, non--trivial.

In this section, we introduce some refinements to Algorithm
\ref{alg1} in order to bias the
construction toward obtaining `large' or `small' caps
containing a prescribed set $\cC$.
\begin{Def}
  Let $\cC$ be a non--empty cap;
  for any $x\in\cU$, the {\em relevance} of $x$
  with respect to $\cC$ is
  \[ r(x,\cC):=|Gx\cup G\cC|-|G\cC|. \]
\end{Def}
Clearly, if $x\in\cC$, then $r(x,\cC)=0$. Hence, when $x\in\cC$,
we shall usually speak of the number
\[ r(x,\cC\setminus\{x\}) \]
as the {\em relevance of $x$ in $\cC$}.
\par
A  dual notion to relevance is that of {\em 
  coverage}.
\begin{Def}
  For any $y\in\cU$, the {\em coverage} of
  $y$ by $\cC$ is the number $c(y,\cC)$ of points
  in
  $x\in\cC$ such that $y\in T_x\cU$.
\end{Def}

The most efficient way to determine
$c(x,\cC)$ is as the cardinality of the set of points of
$\cC$ which are conjugate to $x$.
From
\[ |Gx\cup G\cC|=|Gx|+|G\cC|-|Gx\cap G\cC|, \]
it follows that
\[ r(x,\cC)+c(x,\cC)=|Gx|=q^3+q^2+1. \]
Hence, $r(x,\cC)$ might be computed directly
from $c(x,\cC)$.
\par
\begin{Def} 
  The {\em  weight} of the point $x\in\cC$
  in $\cC$
  is the number
  \[ w(x,\cC):=\sum_{y\in Gx}\frac{1}{c(y,\cC)} \]
\end{Def}
The assumption $x\in\cC$ guarantees
that $c(y,\cC)\neq 0$. 
\par
For any $x\in\cC$, define
\[ \cC_x:=\cC\setminus\{x\}. \]
Then, for any $y\in\cU$, the relevance of $y$ with
 respect to $\cC_x$ may
 be written as
\[ r(y,\cC_x)=r(y,\cC)+|(Gx\cap Gy)\setminus\cC_x|. \]
\begin{Pro}
  \label{pr2}
  Let $x\in\cC$ and assume $y\not\in G\cC$.
  Then,
  \[ |G\cC_x|=|G\cC|-r(x,\cC_x), \]
  and 
  \[ |G(\cC\cup\{y\})|=|G\cC|+r(y,\cC). \]
  Furthermore, $\cC$ is complete if and only if
  $|G\cC|=(q^3+1)(q^2+1)$.
\end{Pro}
The weight of a point $x\in\cC$ and its coverage
by $\cC_x$ are clearly related.

\begin{Pro}
  \label{relat}
  For any $x\in\cC$,
  \[ w(x,\cC)=r(x,\cC_x)+\sum_{y\in Gx\cap G\cC_x}\frac{1}{c(y,\cC_x)+1}. \]
\end{Pro}
\begin{proof}
  If $y\in Gx$, then
  \[ c(y,\cC)=c(y,\cC_x)+1. \]
  For  $y\in Gx\setminus G\cC_x$, the coverage of $y$ by $\cC_x$ is
  $c(y,\cC_x)=0$; hence, $c(y,\cC)=1$. It follows that
  \begin{eqnarray*}
    \displaystyle
    \sum_{y\in Gx\setminus G\cC_x}\frac{1}{c(y,\cC)}&=%
    \displaystyle 
    \sum_{y\in Gx\setminus G\cC_x}1= |Gx\setminus G\cC_x|=\\[.4cm]
    &=|Gx\cup G\cC_x|-|G\cC_x|=r(x,\cC_x).
  \end{eqnarray*}
  This implies
  \begin{eqnarray*}
    w(x,\cC)&=\displaystyle\sum_{y\in Gx\setminus G\cC_x}\frac{1}{c(y,\cC)}+
    \sum_{y\in Gx\cap G\cC_x}\frac{1}{c(y,\cC)}= \\
    &=\displaystyle
    r(x,\cC_x)+\sum_{y\in Gx\cap G\cC_x}\frac{1}{c(y,\cC_x)+1},
  \end{eqnarray*}
  and the result follows.
\end{proof}
A straightforward argument now proves that
\[ r(x,\cC)\geq 2w(x,\cC) - (q^3+q^2+1). \]
\begin{Pro}
  \label{abwh}
  For any complete cap ${\cC}$,
  \[ \sum_{x\in{\cC}}w(x,{\cC})=(q^3+1)(q^2+1). \]
\end{Pro}
\begin{proof}
  Since $\cC$ is complete, the union of all $Gx$, as
  $x$ varies in $\cC$, is $\cU$. Hence, $|G\cC|$ 
  might
  be written as
  \[
  |G\cC|=\sum_{x\in\cC}\sum_{y\in Gx}\frac{1}{c(y,\cC)}= 
  \sum_{y\in\cU}\frac{c(y,\cC)}{c(y,\cC)}= \\
  \sum_{y\in\cU}1=|\cU|.
  \]
  The proposition follows.
\end{proof}

\begin{Pro}
  \label{abwhc}
  Let $\cC$ be a complete cap of cardinality $q^2+1$. Then,
  there exists $x\in\cC$ such that
  \[  w(x,\cC)\geq q^3+1; \]
  likewise, if $\cC$ is an ovoid, 
  then there is $x\in\cC$ such that
  \[ w(x,\cC)\leq q^2+1. \]
\end{Pro}

Proposition \ref{abwhc} can be proved as an immediate corollary of
Proposition \ref{abwh}. It suggests that if a cap is large,
 then its points might be expected to have small weight and that,
 conversely, the weight of points of a large cap is usually
 fairly large.

\begin{Pro}
  \label{pro2}
  Let $\cC$ be a non--empty cap;
  then, for any $x\in\cU$ not covered by $\cC$,
  \[ 1\leq r(x,\cC)\leq q(q^2+q-1). \]
  Furthermore, 
  if there is $x\in\cU$ such that $r(x,\cC)=1$,
  then $|\cC|\geq q^2$.
\end{Pro} 
\begin{proof}
  Clearly, for $\cC\subseteq\cC'$, 
  \[ r(x,\cC)\geq r(x,\cC'). \]
  Hence, in
  order to prove the upper bound on 
  $r(x,\cC)$, it is enough to consider the
  case when $|\cC|=1$.
  Assume $x,y$ be two distinct points of $\cU$ and suppose 
  that $x$ is not covered by $y$. Then, $x\not\in T_y\cU$ and
  neither $x$ nor $y$ are on
  the line
  \[ T_{xy}\cU=T_x\cU\cap T_y\cU. \]
  Furthermore,
  $T_{xy}\cU$ meets $\cU$ in $q+1$ points and
  \[ Gx\cap Gy=Gx\cap T_{xy}\cU. \]
  Hence, 
  \[ |Gx\cap Gy|=q+1. \]
  It follows that
  \[ r(x,\{y\})=q(q^2+q-1). \]
  The lower bound on $r(x,\cC)$ is immediate.
  \par
  Suppose now $r(x,\cC)=1$, and consider a component $L$ of
  $\cU$ which is in $Gx$.
  All points of $L$ but $x$ are covered by some point of $y\in\cC$.
  Hence,
  \[ \forall t\in L\setminus\{x\}, \exists y\in\cC: t\in T_y\cU\cap T_x\cU.\]
  \par
  Furthermore, if two points $t,t'$ of $L$ were covered by the same $y\in\cU$,
  then $tt'=L\subseteq T_y\cU$ and $x$ would also by covered by $y$ --- a
  contradiction, since $r(x,\cC)=1$. This implies that $\cC$ contains at 
  least $q^2$ points.
\end{proof}
\begin{Pro}
  \label{tech-lemma}
  The second largest value for $r(x,\cC)$ is
  $q^3+q^2-2q$.
\end{Pro}
\begin{proof}
  As before, it might be assumed without loss of generality that
  $\cC=\{y,z\}$. Let $x\in\cU\setminus G\cC$.
  Then, either 
  \[ T_{xy}\cU=T_{xz}\cU=T_{xz}\cU=L, \]
  or
  \[ T_{xy}\cU\cap T_{yz}\cU\cap T_{xz}\cU=\{p\}. \]
  In the former case,  
  \[ r(x,\cC)=|T_x\cU\cap\cU|-|L\cap\cU|=q(q^2+q-1). \]
  In the latter, the lines $T_{xy}$, $T_{yz}$ and $T_{xz}$
  are not tangent to the surface $\cU$. Hence, each of
  them meets $\cU$ in $q+1$ points. 
  There are two possibilities:
  \begin{enumerate}
  \item
    if $p\not\in\cU$, then 
    \[ r(x,\cC)=q^2(q+1)+1-2(q+1)=q^3+q^2-2q-1; \]
  \item
    if $p\in\cU$, then 
    \[ r(x,\cC)=q^2(q+1)+1-2q-1=q^3+q^2-2q. \]
  \end{enumerate}
  The result follows
\end{proof}

We adopted two different approaches to the construction of
optimal completions of a partial cap $\cC$:
\begin{enumerate}
\item
  a forward--looking algorithm, in which points to be added are 
  chosen carefully at each iteration;
\item
  a backtracking technique, in which a small completion of the
  original cap, obtained, say, using Algorithm \ref{alg1},
  is enlarged by replacing suitable points.
\end{enumerate}
\subsection{The forward--looking approach}
The main advantage of this approach is that it is possible to estimate
{\em a priori} the complexity and the execution time of the algorithm;
however, unless all possible completions are examined or an 
ovoid is found, we are usually unable to
guarantee that the completion that has been constructed is
actually optimal.
\par
For any cap $\cC$, define two functions
\begin{eqnarray*}
  r^{+}(\cC):=&\max_{x\not\in G\cC}r(x,\cC); \\
  r^{-}(\cC):=&\min_{x\not\in G\cC}r(x,\cC). 
\end{eqnarray*}
Clearly, $r^{-}(\cC)=0$ if and only if $r^{+}(\cC)=0$ and
the cap $\cC$ is complete. One remarkable case arises 
when $r^{+}(\cC)=1$.
\begin{Pro}
  Let $\cC$ be a cap and suppose $r^{+}(\cC)=1$.
  Then, there exists exactly one complete cap $\widetilde{\cC}$
  such that $\cC\subseteq\widetilde{\cC}$ and
  \[ \widetilde{\cC}=\cC\cup(\cU\setminus G\cC). \]
\end{Pro}
\begin{proof}
  Let $\cM=\cU\setminus G\cC$.
  Clearly, if $\cC\cup\cM$ is a cap, then it is complete,
  since all the points of $\cU$ are being covered by it.
  The proof that  $\cC\cup\cM$ is a cap is by induction on $n=|\cM|$.
  \par
  For $n=1$, the proposition is trivial.
  \par
  Assume now $n>1$, and
  let $x$ be a point of $\cM$.
  Since $r^{+}(\cC)=1$, then $r(x,\cC)=1$.
  Define $\cC^{x}=\cC\cup\{x\}$. Clearly $\cC^{x}$ is a cap;
  furthermore, 
  \[ G(\cC^{x})=G\cC\cup\{x\}, \]
  that is
  \[ \cM^{x}:=(\cU\setminus\cC^{x})=\cM\setminus\{x\}. \]
  Hence, $|\cM^{x}|=n-1$ and for any $y\in\cM^{x}$,
  \[ r(y,\cC^{x})=1. \]
  The result now follows from the inductive assumption.
\end{proof}
\begin{Pro}
  The function $r^{+}$ is monotonic non--increasing, in
  the sense that
  \[ \cC'\subseteq\cC \Rightarrow r^{+}(\cC')\geq r^{+}(\cC). \]
\end{Pro}
\begin{proof}
  It is possible to assume without loss of generality
  $\cC'=\cC_{x}$.
  Take $y\in\cU$ to be a point of $\cU\setminus G\cC$ such
  that $r(y,\cC)=r^{+}(\cC)$. Then,
  \[ r^{+}(\cC_x)\geq r(y,\cC_x)=r(y,\cC)+|(Gx\cap Gy)\setminus\cC_x|\geq%
  r^{+}(\cC). \]
  The result follows.
\end{proof}
The simplest  selection technique which can be used in order to
construct large complete caps is {\em to choose at each iteration
  a point in $\cU$ of minimal relevance}, that is $x\in\cU$ such that
\[ r(x,\cC)=r^{-}(\cC). \]
Clearly, this is the choice for a point to be added to $\cC$ which
is `locally best', in the sense that 
it always minimises the number of new covered points.
However, the function $r^{-}(\cC)$ needs not be monotonic and
this approach might leave points of  weight regrettably  large to be
added in the final stages of the construction --- the cap thus
obtained, hence, may not be {\em maximum}.
In order to get further insights on  this issue, the algorithm has been
tested  providing as initial input a small subset of the
points of a known ovoid.
The results of this approach are discussed in Section \ref{sec:32}.
\par
It has been seen that, if the initial datum is small and random, 
then the result is a complete cap of size which usually
approximates $q^3-q^2$.
This confirms that,
 while the biased algorithm  provides caps much larger than the ones
 of Algorithm \ref{alg1}, none the less
 the choice of the point $x$ to be added to $\cC$ at
 each iteration should not depend only on the value of $r^{-}(\cC)$.
\par

\begin{Pro}
  \label{ov1le}
  Let $\cO$ be an ovoid. Then, for any $x\in\cO$, 
  \[ r(x,\cO_x)=1.\]
\end{Pro}
\begin{proof}
  Any point $p\in\cU$ belongs to exactly $q+1$ generators.
  An ovoid
  $\cO$ is a set of $q^3+1$ points which blocks all 
  $(q^3+1)(q+1)$ lines of the Hermitian surface $\cU$; hence, for
  each $[\in\cU$, all generators through $p$ are blocked.
  \par
  Assume now that $r(x,\cO_x)>1$. Then, there is a point
  $y\in\cU\setminus\cO$ such that $y\in Gx$ and $y\not\in Gz$ for
  any $z\in\cO_x$.
  Clearly, the only line through $y$ which is blocked by $x$
  is $xy$.
  It follows that there are at least $q$ points of $\cO_x$
  which cover $y$ --- a contradiction.
  It follows that $r(x,\cO_x)=1$.
\end{proof}

\begin{Pro}
  \label{ov2le}
  Let $\cO$ be an ovoid. Then, for any $\Omega\subseteq\cO$ such
  that $|\Omega|<q+1$,
  \[ r^{+}(\cO\setminus\Omega)=1. \]
\end{Pro}
\begin{proof}
  Any point $y\in\cU\setminus\cO$ is covered by $q+1$ points
  of $\cO$. Hence, all the points of $\cU\setminus\cO$ are
  covered by the cap $\cO\setminus\Omega$. Since $\Omega$ is a cap,
  it follows that the relevance of each $x\in\Omega$ is  $1$,
  which provides the result.
\end{proof}
An immediate consequence of Proposition \ref{ov2le} is that
if a set $\cC$ of $q^3-q+1$ points is contained in an ovoid $\cO$, then
$\cO$ is the only complete cap containing $\cC$.
\begin{Cor}
  Let $\cO$ and $\cO'$ be two distinct ovoids. Then,
  \[ |\cO\setminus\cO'|\geq q+1. \]
\end{Cor}
There are complete ovoids which differ in exactly $q+1$ points; for
instance, this is the case for ovoids obtained from 
each other by derivation, see
\cite{PT}.
\begin{Pro}
  \label{ov3le}
  Let $\cO$ be an ovoid. Then, there is $\Omega\subseteq\cO$ such
  that $|\Omega|\geq\frac{1}{2}(q^2+q)$ and the only complete cap containing
  $\cO':=\cO\setminus\Omega$ is $\cO$.
\end{Pro}
\begin{proof}
  The set $\Omega$ will be constructed step by step.
  Let $P_0$ be any point of $\cU\setminus\cO$; then, $P_0$ is covered
  by $q+1$ points of $\cO$. Take now as
  $\Omega_0$ any set of $q$ points of $\cO$ covering $P_0$ and
  let
  \[ \Lambda_1:=\cO\setminus\Omega_0. \]
  From Proposition \ref{ov2le}, 
  the only complete cap
  containing $\Lambda_1$ is $\cO$. \par
  For each $q>i>0$, fix a point $P_i$ in 
  $\cU\setminus\cO$ such that 
  $P_i$ is covered by at least $q+1-i$ points of
  \[ \Lambda_i:=\Lambda_{i-1}\setminus\Omega_{i-1}. \]
  Observe that any point of $\cU\setminus\cO$ different from
  the $P_j$'s with $j<i$ satisfies this condition.
  Then, let $\Omega_i$ be a set of $q-i$ points of
  $\Lambda_i$ covering $P_i$ --- it follows that
  $\Omega_i$ is, by construction, disjoint from any
  of the $\Omega_j$ for $j<i$. 
  This procedure may be iterated $q$ times.
  Define now
  \[ \Omega:=\bigcup_{i=0}^{q-1}\Omega_i. \]
  Since, for $i\neq j$, 
  \[ \Omega_i\cap\Omega_j=\emptyset, \]
  the cardinality of $\Omega$ is $\frac{1}{2}q(q+1)$.
  Furthermore, each point of 
  $\cU\setminus\cO$ is covered by $\cO'$.
  It follows that any completion of $\cO'$ is contained
  in $\cO'\cup\Omega$.
  The result is now a consequence of the fact that
  $\cO'\cup\Omega$ is a complete cap.
\end{proof}
Propositions \ref{ov1le}, \ref{ov2le} and \ref{ov3le} suggest that a
 a partial cap $\cC$ of size approximately $q^3-q^2$ could be enough to
 determine an ovoid. However,  in order for such a set $\cC$ to be contained
 in an ovoid, it is necessary that many of
 the points of  $\cU\setminus G\cC$ have small relevance.
\par
This inspired the following strategy to look for large caps when
provided only with a small initial datum:
rather than choosing every time a point with the smallest relevance,
it is possible to pick an $x$ which yields a large number of
points of minimal relevance for $\cC^{x}$. 
\par
This approach may be implemented as follows.
Given a cap $\cC$ and a point $x$, define
$\rho^{-}(x,\cC)$ as the number of points $t$ in $\cC_x$ such that
$r(t,\cC_x)=r^{-}(\cC_x)$.
Then,
\[ \rho^{-}(x,\cC):=
|\{t\in\cU: r(t,\cC_x)=r^{-}(\cC_{x})\}|. \]
In Algorithm \ref{alg3}, a point $x$ which maximises $\rho^{-}(\cC)$ is
determined.
The symbol $\oplus$ is used to denote the concatenation of
two ordered lists.
\begalg
\label{alg3}
\begin{tabular}{ll}
  {\bf Input:} & a cap $\cC$; \\
  {\bf Output:}& a point $x\not\in G\cC$.
\end{tabular}
\vskip.2cm
Fw\_Complete:=
\begin{enumerate}
\item if $r^{-}(\cC)=1$, then return any $x\in G\cC$ and exit;
\item $M\leftarrow[\ ]$;
\item For $t\not\in G\cC$,
  \begin{enumerate}
  \item $\cC_{0}\leftarrow \cC\cup\{t\}$;
  \item $L\leftarrow \{x\in\cU: r(x,\cC_{0})=r^{-}(\cC_{0})\}$;
  \item $M\leftarrow M\oplus[L]$;
  \end{enumerate}
\item $k\leftarrow\min \{|L|:L\in M\}$;
\item select $x\in G\cC$ such that 
  $\rho^{-}(x,\cC)=k$.
\end{enumerate}
\caption{Point selection: forward search}
\enalg
\subsection{The backtracking approach}

\begin{Pro}
  \label{sugg}
  Let $\cC$ be a complete cap of cardinality $n$ and
  assume that there is $p\in\cC$ such that
  for some $x\in Gp\setminus G\cC_p$,
  \[ r(p,\cC_p)>r(x,\cC_p). \]
  Then, the cap $\cC_p$ is contained in a complete cap of cardinality
  at least $n+1$.
\end{Pro}
\begin{proof}
  From Proposition \ref{pr2},
  \[ |G(\cC_p\cup\{x\})|=|G\cC|-r(p,\cC_p)+r(x,\cC_p). \]
  Since $r(x,\cC_p)<r(p,\cC_p)$, it follows that
  \[ |G(\cC_p\cup\{x\})|<(q^3+1)(q^2+1).\]
  Hence, $\cC_p\cup\{x\}$ is a cap of cardinality $n$ which is not
  complete and contains $\cC$. The result follows.
\end{proof} 

Another way to construct large caps is, as
Proposition \ref{sugg}
suggests,
by a backtracking procedure.
The main idea underlying this technique is to start with a small
complete cap $\cC$ and try to replace points with large relevance with
others whose relevance is smaller.
\par
In general, it might not be possible to find a good replacement if
only one point is removed; this is the case, for example,
 when the starting cap is already fairly large.

For instance,
according to Proposition \ref{ov2le}, if 
a cap $\cC$ has size is at least $q^3-q+1$ and it
is contained in an ovoid $\cO$, then all the points which are not covered
by $\cC$ have relevance $1$. 
Clearly, in order to succeed, the algorithm needs to remove
as many points from the cap as to be able to fine
some point which is
not covered anymore and that has relevance larger than $1$.
\par
However,
as the following propositions show, 
it has to be expected that very few points of a minimal
complete cap have small relevance. Furthermore, 
if any point of a complete cap $\cC$ has large relevance, then
it is always possible to construct another complete cap $\cC'$ in such 
a way as to have
$|\cC\setminus\cC'|=1$ and $|\cC'|>|\cC|+1$.
\begin{Pro}
  Let $\cC$ be a complete cap and assume that there is $p\in\cC$ such
  that $r(p,\cC_p)>q^2+1$.
  Then, for any  $x\in \Gamma_p:=Gp\setminus (G\cC_p\cup\{p\})$,
  \[ r(x,\cC_p)<r(p,\cC_p). \]
\end{Pro}
\begin{proof}
  Since $r(p,\cC_p)>q^2+1$, not all the points of $\Gamma_p$ lie on a
  line.
  On the other hand, for any $x\in\Gamma_p$,
  \[ Gp\cap Gx=px \]
  Let now $\cC'=\cC_p\cup\{x\}$.
  From the first remark above, there is 
  $y\in Gp\setminus Gx$ such that
  \[ y\not\in G(\cC')=G\cC_p\cup px. \]
  Since $\cC$ is complete, 
  \[ Gx\setminus G\cC_p= Gp\cap Gx=px. \]
  From this, the result follows and
  \[ r(x,\cC'_x)=r(x,\cC_p)\leq q^2+1.\]
\end{proof}
\begin{Pro}
  \label{obsr}
  Let $\cC$ be a complete cap of cardinality $q^2+1$. Then, there is
  $p\in\cC$ such that $r(p,\cC_p)>q^2+1$.
\end{Pro}
\begin{proof}
  Suppose that $r^{+}(\cC)<q^2+1$.
  Then,
  \[ (q^3+1)(q^2+1)=|\cU|\leq (q^2+1)r^{+}(\cC)\leq (q^2+1)^2, \]
  a contradiction.
\end{proof}
A simple backtracking approach is presented in Algorithm \ref{alg4}.
Proposition \ref{obsr} guarantees that, given any cap $\cC$, 
a point is determined after at most $|\cC|-q^2-1$ recursive calls.

\begalg
\label{alg4}
\begin{tabular}{ll}
  {\bf Input:} & a cap $\cC$, a cap $\cC'$ with $\cC\subseteq\cC'$; \\
  {\bf Output:}& a cap $\cC''$ with $\cC\subseteq\cC''$.
\end{tabular}
\vskip.2cm
Large\_Cap($\cC$,$\cC'$):=
\begin{enumerate}
\item if $\cC'=\cC''$, then exit;
\item compute 
  $M=\max_{t\in \cC'\setminus\cC} r(t,\cC')$;
\item select $x\in\cC'\setminus\cC$ such that
  $r(p,\cC')=M$;
\item $\cC''\leftarrow\cC'\setminus\{p\}$;
\item if $\exists x\not\in G\cC''$ such that
  $r(x,\cC''\cup\{x\})<M$, then 
  \begin{enumerate}
  \item $\cC''\leftarrow\cC''\cup\{x\}$;
  \item return $\cC''$;
  \end{enumerate}
  else
  \begin{enumerate}
  \item $\cC''\leftarrow \text{Large\_Cap}(\cC,\cC'')$;
  \end{enumerate}
\item let $x\not\in G\cC''$ such that
  \[ w(x,\cC''\cup\{x\})=\min_{y\not\in G\cC''}w(y,\cC''\cup\{y\}); \]
\item $\cC''\leftarrow \cC''\cup\{x\}$.
\end{enumerate}
\caption{Backtracking: large caps}
\enalg 

\section{Results of Algorithm \ref{alg1}}
Algorithm \ref{alg1}, as presented in this paper, has been implemented
with the
computer algebra package GAP \cite{GAP} and some tests have
been performed for small values of $q$, namely $q=5$ and $q=7$.
The methodology followed has usually
been to iterate each test at least $1000$  times and then consider
an average of the results.
The numbers in all the tables of this section represent the {\em chance}
 of obtaining a cap of given size using the algorithm, when a set of
 prescribed cardinality
 is provided as input.

\subsection{Random search}
\label{sec:31}
Algorithm \ref{alg1}, with the selection of points done at random,
may be used in order to investigate the {\em spectrum} of
complete caps of the Hermitian surface.
The results of a test performed with the empty set as initial datum
are presented in Table \ref{a1q5t1} for $q=5$ and in Table \ref{a1q7t1}
for $q=7$.
\begin{table}[ht]
  \[ \begin{array}{r|D{.}{.}{1}}
    \hline
    |\ccC| & \multicolumn{1}{c}{\%}   \\
    \hline
    78  & 0.1  \\
    79  & 1.0  \\
    80  & 1.9  \\
    81  & 5.9  \\
    82  & 9.3  \\
    83  & 16.3  \\
    84  & 19.7  \\
    \hline\end{array} \qquad\qquad
  \begin{array}{r|D{.}{.}{1}}
    \hline
    |\ccC| & \multicolumn{1}{c}{\%}   \\
    \hline
    85  & 16.5  \\
    86  & 12.6  \\
    87  & 9.5 \\
    88  & 4.7 \\
    89  & 1.6 \\
    90  & 0.8  \\
    91  & 0.1   \\
    \hline
  \end{array} \]
  \caption{Distribution of caps:
    results of Algorithm \ref{alg1} with $q=5$ and $\cC=\emptyset$}
  \label{a1q5t1}
\end{table}
\begin{table}[ht]
  \[ \begin{array}{r|D{.}{.}{1}}
    \hline
    |\ccC| & \multicolumn{1}{c}{\%}   \\
    \hline
    195 & 0.3 \\
    196 & 0.6 \\
    197 & 1.4 \\
    198 & 2.0 \\
    199 & 4.0 \\
    200 & 5.8 \\
    201 & 8.7 \\
    202 & 10.3 \\
    \hline\end{array} \qquad\qquad
  \begin{array}{r|D{.}{.}{1}}
    \hline
    |\ccC| & \multicolumn{1}{c}{\%}   \\
    \hline
    203 & 10.3 \\
    204 & 11.1 \\
    205 & 12.8 \\
    206 & 9.0 \\
    207 & 8.1  \\
    208 & 7.7  \\
    209 & 4.6  \\
    210 & 1.9  \\
    \hline\end{array} \qquad\qquad
  \begin{array}{r|D{.}{.}{1}}
    \hline
    |\ccC| & \multicolumn{1}{c}{\%}   \\
    \hline
    211 & 0.1  \\
    212 & 1.0  \\
    213 & 0.3  \\
    \hline
  \end{array} \]
  \caption{Distribution of caps:
    results of Algorithm \ref{alg1} with $q=7$ and $\cC=\emptyset$}
  \label{a1q7t1}
\end{table}

The same algorithm, for $q=5$, when the input $\cC$ has been
a set  of $50$ random points contained
in an ovoid, has produced at least one large complete cap, but no
ovoid, as it can be seen in Table \ref{a1q5t2}. 
The results for $q=7$ with an input set $\cC$ of size $98$ have
been similar, see Table \ref{a1q7t2}.
\begin{table}[ht]
  \[ \begin{array}{r|D{.}{.}{1}}
    \hline
    |\ccC| & \multicolumn{1}{c}{\%}   \\
    \hline
    81  & 0.1  \\
    82  & 0.2  \\
    83  & 0.7 \\
    84  & 1.0  \\
    85 & 2.4 \\
    86  & 4.8 \\ 
    87  & 7.9  \\
    88 & 12.7  \\
    89 & 10.3  \\
    90 & 12.4 \\
    91 & 10.9 \\ \hline
  \end{array}\qquad\qquad
  \begin{array}{r|D{.}{.}{1}}
    \hline
    |\ccC| & \multicolumn{1}{c}{\%}   \\
    \hline
    92 & 10.3 \\
    93  & 6.5  \\
    94 & 5.9  \\
    95 & 4.3   \\
    96 & 2.8   \\
    97 & 3.0  \\
    98 & 1.4 \\
    99 & 1.0 \\
    100 & 0.4 \\
    101 & 0.4 \\
    102 & 0.2 \\ \hline
  \end{array}\qquad\qquad
  \begin{array}{r|D{.}{.}{1}}
    \hline
    |\ccC| & \multicolumn{1}{c}{\%}   \\
    \hline
    103 & 0.2 \\
    104 & 0.1 \\
    106 & 0.1  \\
    112 & 0.1 \\
    \hline
  \end{array} \]
  \caption{Distribution of caps:
    results of Algorithm \ref{alg1} with $q=5$ and $|\cC|=50$}
  \label{a1q5t2}
\end{table}
\begin{table}[ht]
  \[ \begin{array}{r|D{.}{.}{1}}
    \hline
    |\ccC| & \multicolumn{1}{c}{\%}   \\
    \hline
    198 & 0.2 \\
    199 & 0.5 \\
    200 & 0.1 \\
    201 & 1.9 \\
    202 & 2.8 \\
    203 & 4.0 \\
    204 & 5.0 \\
    205 & 8.6 \\ 
    206 & 9.2 \\
    207 & 11.2 \\
    208 & 13.1 \\
    \hline
  \end{array}\qquad\qquad
  \begin{array}{r|D{.}{.}{1}}
    \hline
    |\ccC| & \multicolumn{1}{c}{\%}   \\
    \hline
    209 & 11.1 \\
    210 & 9.3 \\
    211 & 7.3 \\ 
    212 & 7.0 \\
    213 & 3.6 \\ 
    214 & 2.5 \\
    215 & 1.0 \\
    216 & 0.8 \\
    217 & 0.5 \\
    218 & 0.2 \\
    219 & 0.1 \\
    \hline
  \end{array}
  \]
  \caption{Distribution of caps:
    results of Algorithm \ref{alg1} with $q=7$ and $|\cC|=98$}
  \label{a1q7t2}
\end{table}
Clearly, as it had to be expected, ovoids represent only a tiny fraction
of possible complete caps and it is very difficult for them to occur
if the initial cap has size much smaller
than $q^3-q^2$.
However, 
as the size of the  input set grows, the chances for a `random'
completion of the cap to be an ovoid
increase as well: this  can be seen in Table
\ref{a1q5t3}, where
the results of an experiment realised with $|\cC|=69$ and $q=5$ are presented.
\begin{table}[ht]
  \[ \begin{array}{r|D{.}{.}{1}}
    \hline
    |\ccC| & \multicolumn{1}{c}{\%}   \\
    \hline
    100  & 0.1  \\
    101 &  0.6 \\
    102  & 0.5   \\
    103  & 0.9  \\
    104  & 1.0  \\
    105  & 1.1  \\
    106 &  1.8  \\
    107 &  1.1  \\
    108 &  3.5  \\
    109 &  3.3  \\
    110 &  4.8  \\
    \hline\end{array} \qquad\qquad
  \begin{array}{r|D{.}{.}{1}}
    \hline
    |\ccC| & \multicolumn{1}{c}{\%}   \\
    \hline
    111  &  4.8  \\
    112  &  6.7 \\
    113  &  6.0 \\
    114  &  3.0  \\
    115  &  2.7  \\
    116  & 14.4  \\
    117  & 9.0 \\
    118  & 1.9 \\
    119  & 8.0 \\
    121  & 22.4  \\
    126  & 9.7   \\
    \hline
  \end{array} \]
  \caption{Distribution of caps:
    results of Algorithm \ref{alg1} with $q=5$ and $|\cC|=69$}
  \label{a1q5t3}
\end{table}
\begin{table}[ht]
  \[ \begin{array}{r|D{.}{.}{1}}
    \hline
    |\ccC| & \multicolumn{1}{c}{\%}   \\
    \hline
    291 & 0.1 \\
    293 & 0.1 \\
    296 & 0.1 \\
    299 & 0.1 \\
    300 & 0.2 \\
    301 & 0.3 \\
    302 & 0.1 \\
    303 & 0.2 \\
    304 & 0.3 \\
    305 & 0.3 \\
    306 & 0.4 \\
    307 & 0.7 \\
    308 & 1.0 \\
    \hline\end{array} \qquad
  \begin{array}{r|D{.}{.}{1}}
    \hline
    |\ccC| & \multicolumn{1}{c}{\%}   \\
    \hline
    309 & 0.5 \\
    310 & 0.9 \\
    311 & 1.2 \\
    312 & 1.4 \\
    313 & 1.7 \\
    314 & 2.0 \\
    315 & 0.9 \\
    316 & 1.0 \\
    317 & 2.3 \\
    318 & 4.4 \\
    319 & 3.7 \\
    320 & 1.6 \\
    321 & 0.8 \\
    \hline
  \end{array}\qquad
  \begin{array}{r|D{.}{.}{1}}
    \hline
    |\ccC| & \multicolumn{1}{c}{\%}   \\
    \hline
    322 & 0.3 \\
    323 & 5.6 \\
    324 & 8.5 \\
    325 & 5.3 \\
    326 & 2.2 \\
    327 & 1.2 \\
    328 & 0.3 \\
    329 & 0.1 \\
    330 & 15.3 \\
    331 & 6.4 \\
    332 & 0.9 \\
    333 & 0.2 \\
    335 & 0.2 \\ 
   \hline
  \end{array}\qquad
  \begin{array}{r|D{.}{.}{1}}
    \hline
    |\ccC| & \multicolumn{1}{c}{\%}   \\
    \hline
    337 & 19.4 \\
    344 & 7.6 \\
    \hline
  \end{array} \]
  \caption{Distribution of caps:
    results of Algorithm \ref{alg1} with $q=7$ and $|\cC|=190$}
  \label{a1q7t3}
\end{table}
  \begin{table}[ht]
    \[ \begin{array}{r|D{.}{.}{1}}
      \hline
      |\ccC| & \multicolumn{1}{c}{\%}   \\
      \hline
      331 & 2.0 \\
      337 & 17.0 \\
      344 & 81.0 \\
      \hline\end{array} \qquad\qquad
    \]
    \caption{Distribution of caps:
      results of $100$ runs of Algorithm \ref{alg1} with $q=7$ and $|\cC|=237$}
    \label{a1q7t4}
  \end{table}
  \begin{table}[ht]
 \[ \begin{array}{r|D{.}{.}{1}}
      \hline
      |\ccC| & \multicolumn{1}{c}{\%}   \\
      \hline
      705 & 1.0 \\
      712 & 3.0 \\
      713 & 4.0 \\
      720 & 1.0 \\
      721 & 24.0 \\
      730 & 67.0 \\
      \hline\end{array} \qquad\qquad
    \]
    \caption{Distribution of caps:
      results of $100$ runs of Algorithm \ref{alg1} with $q=9$ and $|\cC|=450$}
    \label{a1q9t4}
    \end{table}
  Observe that no caps with size $121<|\cC|<126$ have been found. 
  The same computations for $q=7$ and $|\cC|=190$ provide the results
  of Table \ref{a1q7t3}.
  The same tests, when performed for $q=7$ and $|\cC|=237$ and for
  $q=9$ and $|\cC|=450$ have
 produced the results of Tables \ref{a1q7t4} and
 \ref{a1q9t4} --- as it can be seen,
 in this cases most of the complete caps constructed have been
 ovoids.
\par
  However, no complete cap $\cC$ with
  cardinality 
  \[ q^3-q+1<|\cC|<q^3+1 \]
  has been found.
  This suggests the following conjecture.
  \begin{Conj}
    The size of the second largest complete cap of the Hermitian
    surface is $q^3-q+1$.
  \end{Conj}
  \subsection{Biased search}
  \label{sec:32}
  In this subsection we consider complete caps obtained by
  using a variant of Algorithm \ref{alg1}, in which the
  point to be added to the partial cap $\cC$ at each iteration is
  required to have minimal relevance. 
  The initial input, as before, is a partial cap $\cC$ provided by
  a random subset of given size of an ovoid.
  This version of the algorithm has shown an interesting behaviour:
  when the initial {\em datum} is large enough, say $|\cC|>34$ for $q=5$,
  the  the result turns out to be usually, but not always, an ovoid ---
  this proves that this 
  procedure is a definite improvement over the purely random search,
  where, in order to have a reasonable chance of finding ovoids,
  at least $60$ points had to be prescribed. 
  \par
  In order to be able to compare these results with those of the previous
  subsection, we have run the algorithm with $100$ different random 
  subsets of size $69$ as input: the algorithm has, in all these
  cases, constructed an ovoid.
  As a matter of fact, an ovoid has been found even with
  an input {\em datum} as small as a set of only $10$ points. 
  However, we have also verified that there exist caps 
  of size at least $34$ for which this program
  produces completions of size $98$.
 \par
 The results for $q=7$ and $|\cC|\geq 90$ have been similar.
\par
  A future development of this work will be deeper investigation
  of these issues and their  relationship with the structure
  of the original ovoid $\cO$ as described by its group of automorphisms.

\end{document}